\documentclass{amsart}
\usepackage{mathrsfs}

\renewcommand\rightmark{\sc{A unified approach to the theory of connections in Finsler Geometry}}

\newtheorem{theorem}{Theorem}[section]
\newtheorem{lemma}[theorem]{Lemma}
\newtheorem{proposition}[theorem]{Proposition}
\newtheorem{corollary}[theorem]{Corollary}
\newtheorem{propdef}[theorem]{Proposition-Definition}

\theoremstyle{definition}
\newtheorem{definition}[theorem]{Definition}

\theoremstyle{remark}
\newtheorem{remark}[theorem]{Remark}

\numberwithin{equation}{section}

\begin{document}

\title{A unified approach to the theory of connections in Finsler Geometry}

\author[H. Vit\'orio]{Henrique Vit\'orio}
\address{Departamento de Matem\'atica, \hfill\break\indent
Universidade Federal de Pernambuco, \hfill\break\indent
Cidade Universit\'aria, \hfill\break\indent
Recife, Pernambuco, Brazil}
\email{henriquevitorio@dmat.ufpe.br}




\thanks{This work was supported by {\rm CNPq}, grant {\rm No.} 232664/2014-5}




\begin{abstract}
We propose a unified approach to the theory of connections in the geometry of sprays and Finsler metrics which, in particular,
gives a simple explanation of the well-known fact that all the classical Finslerian connections provide exactly the same formulas appearing in
the calculus of variations.


\end{abstract}

\maketitle

\section{Introduction}
Let us start by recalling the following
\begin{definition}
A Finsler metric on a manifold $M$ is a function $F:TM\rightarrow[0,\infty[$, smooth on the slit tangent bundle $TM\backslash 0$, such that
\begin{enumerate}
\item $F(u)=0$ if, and only if, $u=0$;
\item $F(\lambda u)=\lambda F(u)$ whenever $\lambda$ is a positive real number;
\item At each $w\in TM\backslash 0$, the fiber-Hessian of $F^2$ is positive-definite; that is, the {\it fundamental tensor} $g_w:T_xM\times T_xM\rightarrow\mathbb{R}$ ($w\in T_xM$)
of $F$, defined by
\begin{equation}\nonumber
g_w(u,v)=\frac{1}{2}\frac{\partial^2}{\partial t\partial s}\Big|_{(0,0)}F(w+su+tv)^2,
\end{equation}
is a positive-definite inner product on $T_xM$.
\end{enumerate}
\end{definition}
The geodesics of a Finsler metric can be defined as the critical points of the energy functionals
\begin{equation}\label{energyfunctional2}
{\rm E}~:~\Omega_{p_1,p_2}\longrightarrow\mathbb{R}~,~~~{\rm E}(\lambda)=\frac{1}{2}\int_0^1F(\dot{\lambda}(t))^2dt,
\end{equation}
where $\Omega_{p_1,p_2}$ is the space of all regular curves (or piecewise regular) joining $p_1$ to $p_2$. If one aims to compute the second variation
of (\ref{energyfunctional2}), or of its two end-manifold analog (\ref{energyfunctional}), at a critical point, it becomes convenient to dispose of a somehow compatible theory of connections. On the other hand,
the extra dependece on directions for the elements involved (caused by the lack of smoothness of $F$ at the null section) implies that the linear connections
that arise are naturally defined on (a vector subbundle of) the double tangent bundle of $M$ and there is not a canonical choice of such connection. Instead, there are in the literature many connections, each of them
characterized by some kind of compatibility conditions, the most notable ones being due to Berwald, Cartan, Chern and Rund, and Hashiguchi; see, for instance, \cite{zadeh68}, \cite{chern}, \cite{bucataru}, \cite{grifone2}.
We propose here a unified approach to the theory of connections which, in particular, makes clearer the well-known fact that all those classical connections provide exactly 
the same formulas appearing in the calculus of variations. Indeed, we exhibit a mild set of compatibility conditions under which any connection would work just as well.
\par A linear connection on the double tangent bundle of $M$ may, to a great extent, be thought of as a family of affine connections on $M$, and this has the advantage
of avoiding having to work with tensors on the tangent bundle and, thus, making some computations more similar to the way they are done in Riemannian geometry. Indeed, some authors 
have approached Finsler geometry through that point of view; see \cite{mathias}, \cite{rademacher}, \cite{shen} and \cite{miguel}. The present article works out in full generality the transition from linear connections
on the double tangent bundle to families of affine connections, stressing out what conditions a linear connection has to fulfil so as to assure that its corresponding family of affine
connections enjoy the nice properties that ultimately lead to the desired computations. For instance, as we show in $\S$\ref{curvatureendomsub}, a mild nullity
condition on the torsion suffices to guarantee that, under suitable assumption, the corresponding affine connections recover the curvature endomorphism of a spray (and, hence, the flag curvature
of a Finsler metric); this should be compared with \cite{miguel}, where the author established the same result in the particular case of the Chern-Rund affine connections. 
Besides, our approach clarifies why some notions such as the covariant derivative, the curvature endomorphism and the second fundamental form
appearing in (\ref{eq700}) do not depend upon the choice of the connection used to define them (see Corollary \ref{corollarycurvendospray}, Proposition \ref{propfundamental}
and Proposition \ref{propfundamentalform}). All these are carried out in $\S$\ref{sectionsprays} and $\S$\ref{sectionfinsler}, where we have
followed a coordinate-free approach in order to make the computations more geometric.
\par In $\S$\ref{sectionclassicalconnections}, we discuss briefly the already mentioned Finslerian classical connections. In particular we describe the families of affine connections 
corresponding to the Cartan and Hashiguchi linear connections, 
something that apparently was missing from
the literature. These descriptions lead to the remarkable conclusion that the Berwald and Hashiguchi linear connections induce the same family of affine connections,
and the same is true of the linear connections of Cartan and Chern-Rund. We acknowledge that this fact was first noticed by M. Javaloyes and communicated to the author. 
\par We conclude these notes with an appendix in which we show a symplectic, connection-free, definition of the second fundamental form introduced in $\S$\ref{subsectionsecondvariation}.

\section{Sprays, connections and curvature}\label{sectionsprays}
We begin by considering the theory of linear connections, and
their corresponding affine connections, with regard to a given spray on the manifold $M$. In $\S$3, all this formalism will be applied to the case of the geodesic spray of a
Finsler metric $F$.

\subsection{Notations}
The following notations and definitions will be used throughout this work:
\begin{itemize}
\item Given an open set $\mathcal{O}\subseteq M$ and a smooth curve $\lambda:I\subseteq\mathbb{R}\rightarrow M$, $\mathfrak{X}(\mathcal{O})$ and $\mathfrak{X}(\lambda)$ will denote, respectively,
the space of smooth vector fields on $\mathcal{O}$ and along $\lambda$.
\item $\pi:TM\backslash 0\rightarrow M$ will denote the tangent bundle without the null section. 
\item The {\it vertical distribution} on $TM\backslash 0$ is $w\mapsto \mathscr{V}_wTM={\rm ker}({\rm d}\pi(w))$. Vectors (or vector fields) tangent to $\mathscr{V}TM$ are called vertical.
\item The {\it vertical tangent bundle}
\begin{equation}\label{verticalbundle}
p:\mathscr{V}TM\rightarrow TM\backslash 0 
\end{equation}
is obtained by restricting to $\mathscr{V}TM$ the projection map $T(TM\backslash 0)\rightarrow TM\backslash 0$.
\item The {\it vertical lift at $w\in TM\backslash 0$} is the tautological isomorphism 
\begin{equation}\nonumber
i_w:T_{\pi(w)}M\rightarrow\mathscr{V}_wTM~,~~i_w(u)=(d/dt)|_{t=0}(w+t\cdot u).
\end{equation}
Through these maps, to any vector field $U$ defined along a map $f:\Sigma\rightarrow M$, we can associate a vector field $U^\mathfrak{v}$ defined along any 
lift $\overline{f}:\Sigma\rightarrow TM\backslash 0$
of $f$ (i.e. $\pi\circ\overline{f}=f$), called the {\it vertical lift of $U$} ({\it along $\overline{f}$}).
\item The {\it canonical vector field} on $TM\backslash 0$ is the vertical vector field $C$ defined by $C(w)=i_w(w)$.
\item The {\it almost-tangent structure} of $TM\backslash 0$ is the vector-bundle endomorphism
\begin{equation}\nonumber
\mathscr{J}:T(TM\backslash 0)\rightarrow T(TM\backslash 0)~,~~\mathscr{J}(X)=i_w({\rm d}\pi(X)) 
\end{equation}
for $X\in T_w(TM\backslash 0)$. Note that any vertical smooth vector field on $TM\backslash 0$ (i.e. any smooth section of (\ref{verticalbundle})) can be written, non-uniquely, as $\mathscr{J}(X)$
for some $X\in\mathfrak{X}(TM\backslash 0)$. 
\end{itemize}

\subsection{Sprays and connections}

\begin{definition}
A second order differential equation on $M$ is a smooth vector field $S$ on $TM\backslash 0$ such that $\mathscr{J}(S)=C$. This means that the integral curves
of $S$ are of the form $t\mapsto\dot{\gamma}(t)$, for some class of curves $\{\gamma\}$ in $M$ called the geodesics of $S$. If furthermore $[C,S]=S$, then $S$ is called a spray.
\end{definition}
For future reference, we state the following straightforward lemma (see \cite{grifone1}).

\begin{lemma}\label{lemmaspray}
If $S$ is a second order differential equation on $M$ and $X\in\mathfrak{X}(TM\backslash 0)$ is vertical, then $\mathscr{J}[X,S]=X$.
\end{lemma}

\begin{definition}
A connection on $M$ in the sense of Grifone, or simply a connection on $M$, is a smooth 1-form $\Gamma$ on $TM\backslash 0$ with values in $TM$ such that $-\Gamma$ is a
reflexion across $\mathscr{V}TM$, i.e. $\Gamma^2={\bf I}$ and ${\rm ker}(\Gamma+{\bf I})=\mathscr{V}TM$. We say that $\Gamma$ is 1-homogeneous if $[C,\Gamma]=0$.
\end{definition}

A connection $\Gamma$ on $M$ determines an Ehresmann connection on $\pi:TM\backslash 0\rightarrow M$ by defining $\mathscr{H}TM={\rm ker}(\Gamma-{\bf I})$, so that
\begin{equation}\label{decomposition}
T(TM\backslash 0)=\mathscr{H}TM\oplus\mathscr{V}TM.
\end{equation}
The corresponding projection operators will be denoted by $\mathscr{H}$ and $\mathscr{V}$, and vectors (resp. vector fields) tangent to $\mathscr{H}TM$ will be called {\it horizontal}.
We denote by $U^\mathfrak{h}$ the {\it horizontal lift} of a vector field $U\in\mathfrak{X}(M)$, which is the horizontal vector field on $TM\backslash 0$ such that ${\rm d}\pi(U^\mathfrak{h})=U$.
\\\\
By a result of Grifone \cite{grifone1}, a spray canonicaly determines a connection on the manifold:\footnote{We refer to \cite{grifone1} for the definition of a symmetric connection.}

\begin{theorem}
A spray $S$ on $M$ determines a unique symmetric and 1-homogeneous connection $\Gamma$ relatively to which $S$ is horizontal. It is called the {\sl canonical connection} associated to $S$ and
is given by $\Gamma_S=-[S,\mathscr{J}]$.\footnote{Here, $[S,\mathscr{J}]$ stands for the Lie derivative of the tensor field $\mathscr{J}$ along $S$.}
\end{theorem}
\textsl{Until the end of this section, we let be fixed a spray $S$ on $M$ with associated canonical connection $\Gamma_S$}.

\subsection{Linear connections on the vertical tangent bundle}\label{linearconnectionsub}
Following Grifone \cite{grifone2}, we define

\begin{definition}
Let $\nabla$ be a linear connection on $p:\mathscr{V}TM\rightarrow TM\backslash 0$. Then,
\begin{enumerate}
\item We call $\nabla$ {\it almost-projectable} if $\nabla_XC=X$ for all $X$ vertical.
\item We say that $\nabla$ {\it projects onto} $\Gamma_S$, or that it is a {\it lift} of $\Gamma_S$, if $\nabla$ is almost-projectable and
${\rm ker}(\varphi)=\mathscr{H}TM$, where
\begin{equation}\nonumber
\varphi: T(TM\backslash 0) \rightarrow\mathscr{V}TM~,~~\varphi(X)=\nabla_XC.
\end{equation}
Thus, $\nabla$ is a lift of $\Gamma_S$ precisely when $\nabla_XC=\mathscr{V}X$, for all $X\in T(TM\backslash 0)$.
\end{enumerate}
\end{definition}

The appropriate notion of torsion for linear connections on (\ref{verticalbundle}) is given by the following (see \cite{zadeh68}).
\begin{definition}
The {\it torsion} of a linear connection $\nabla$ on $p:\mathscr{V}TM\rightarrow TM\backslash 0$ is the $\mathscr{V}TM$-valued tensor field on $TM\backslash 0$ given by
\begin{equation}\nonumber
{\rm T}(X,Y)=\nabla_X\mathscr{J}(Y)-\nabla_Y\mathscr{J}(X)-\mathscr{J}[X,Y]~,~~X,Y\in\mathfrak{X}(TM\backslash 0).
\end{equation}
\end{definition}

\subsubsection{Nullity conditions for the torsion}
We will be interested in the lifts $\nabla$ of $\Gamma_S$ whose torsion satisfies one of the following conditions (actually, the first one will suffice; we have included 
the other ones as they are fulfilled by the classical connections discussed in $\S$\ref{sectionclassicalconnections})
\begin{eqnarray}
{\bf (T1)} & & \mbox{${\rm T}(S,Y)=0$ for all $Y$.}\nonumber\\
{\bf (T2)} & & \mbox{${\rm T}(\mathscr{H}X,\mathscr{H}Y)=0$ for all $X,Y$.}\nonumber\\
{\bf (T3)} & & \mbox{${\rm T}(X,Y)=0$ for all $X,Y$.}\nonumber
\end{eqnarray}

\noindent Next we show that, {\it as long as {\bf (T1)} holds, applying $\nabla$ in the direction of $S$ does not depend on the choice of $\nabla$.}

\begin{lemma}\label{lemmafundamental}
Let $\nabla$ be a lift of $\Gamma_S$.
\begin{enumerate}
\item If $Y$ is vertical, ${\rm T}(S,Y)=0$. Therefore, ${\bf (T2)}\Rightarrow{\bf (T1)}$.
\item If {\bf (T1)} holds, then $\nabla_S\mathscr{J}(Y)=\mathscr{V}[S,\mathscr{J}(Y)]$ for all $Y\in\mathfrak{X}(TM\backslash 0)$.
\end{enumerate}
\end{lemma}
\begin{proof}
(1) From
$\mathscr{J}(Y)=0$ and $\mathscr{J}(S)=C$ we obtain ${\rm T}(S,Y)=-\nabla_YC-\mathscr{J}[S,Y]$, which, due to the almost-projectability of $\nabla$, is equal to $-Y-\mathscr{J}[S,Y]$. Now the claim follows
from Lemma \ref{lemmaspray}.
(2) From $\nabla_Y\mathscr{J}(S)=\nabla_YC=\mathscr{V}Y$, we obtain $0={\rm T}(S,Y)=\nabla_S\mathscr{J}(Y)-\mathscr{V}Y-\mathscr{J}[S,Y]$. Now, from $\mathscr{V}Y=-\mathscr{V}\Gamma_S(Y)$
and $\Gamma_S=-[S,\mathscr{J}]$,
we obtain, respectively
\begin{eqnarray}
\mathscr{V}Y+\mathscr{J}[S,Y] & = & \mathscr{V}\bigl(-\Gamma_S(Y)+\mathscr{J}[S,Y]\bigr)\nonumber\\
& = & \mathscr{V}\bigl([S,\mathscr{J}](Y)+\mathscr{J}[S,Y]\bigr)\nonumber\\
& = & \mathscr{V}[S,\mathscr{J}(Y)].\nonumber
\end{eqnarray}
\end{proof}
\textsl{Until the end of this section, let be fixed a lift $\nabla$ of $\Gamma_S$.}

\subsubsection{The tensors $\mathscr{C}$ and $\mathcal{C}$}\label{tensorsCsubsub}
When dealing with the transition to the corresponding family of affine connections in $\S$\ref{affineconnectionsub}, it will prove useful 
to consider the following ``vertical part'' of the torsion ${\rm T}$.
\begin{definition}
\begin{enumerate}
\item We define a $\mathscr{V}TM$-valued tensor field $\mathscr{C}$ on $TM\backslash 0$ by
$\mathscr{C}(X,Y)={\rm T}(\mathscr{J}(X),Y)$.
\item For each $w\in TM\backslash 0$, we define $\mathcal{C}_w:T_{\pi(w)}M\times T_{\pi(w)}M\rightarrow T_{\pi(w)}M$ by
$\mathcal{C}_w(u,v)={i_w}^{-1}\mathscr{C}(X,Y)$,
where $X,Y\in T_wTM$ are any lifts of $u,v$. The well-definability of $\mathcal{C}_w$ will follow from the lemma below.
\end{enumerate}
\end{definition}

\begin{lemma}\label{lemmaC} We have
\begin{enumerate}
\item The tensor field $\mathscr{C}$ is semi-basic, i.e, $\mathscr{C}(X,Y)=0$ whenever at least one of $X,Y$ is vertical.
\item $\mathscr{C}(\cdot,S)=0$. Equivalently, $\mathcal{C}_w(\cdot,w)=0$ for all $w$.
\end{enumerate}
\end{lemma}
\begin{proof}
For (1), just note that
\begin{equation}\label{eq3}
\nabla_{\mathscr{J}(X)}\mathscr{J}(Y)=\mathscr{J}[\mathscr{J}(X),Y]+\mathscr{C}(X,Y),
\end{equation}
and that $[\mathscr{J}(X),Y]$ is vertical and, thus, $\mathscr{J}[\mathscr{J}(X),Y]=0$, whenever $Y$ is vertical. 
(2) is just a reformulation of (1) of Lemma \ref{lemmafundamental} since $T$ is skew-symmetric.
\end{proof}

\subsubsection{The curvature endomorphisms of a spray}\label{subsubcurvspray}
Denote by $\mathcal{R}$ the curvature tensor of $\nabla$,
\begin{equation}\nonumber
\mathcal{R}(X,Y)\mathscr{J}(Z)=\nabla_{[X,Y]}\mathscr{J}(Z)-[\nabla_X,\nabla_Y]\mathscr{J}(Z).
\end{equation}
Our aim here is to show that the following definition is intrinsic to the spray $S$.
\begin{definition}
Given $w\in TM\backslash 0$, the {\it curvature endomorphism in the direction $w$} is
the map ${\bf R}_w:T_{\pi(w)}M\rightarrow T_{\pi(w)}M$,
\begin{equation}\nonumber
{\bf R}_w(u)=i_w\hspace{0.01cm}^{-1}\mathcal{R}(w^\mathfrak{h},u^\mathfrak{h})w^\mathfrak{v}=i_w\hspace{0.01cm}^{-1}\mathcal{R}(S,u^\mathfrak{h})C
\end{equation}
\end{definition}

\begin{lemma}\label{lemmacurvatura}
We have
\begin{enumerate}
\item If $X\in\mathfrak{X}(TM\backslash 0)$ is horizontal, then $\mathcal{R}(S,X)C=\mathscr{V}[S,X]$.
\item If {\bf (T1)} holds and if $X$ is vertical, then $\mathcal{R}(S,X)C=0$.
\end{enumerate}
\end{lemma}
\begin{proof}
(1) Let $X\in\mathfrak{X}(TM\backslash 0)$ be horizontal. As $S$ is horizontal as well and $\nabla$ is a lift of $\Gamma_S$, then
$\nabla_SC=\nabla_XC=0$ and $\nabla_{[S,X]}C=\mathscr{V}[S,X]$. Therefore, $\mathcal{R}(S,X)C=\mathscr{V}[S,X]$. (2) If $X$ is vertical,
then $\nabla_XC=X$ and thus $\mathcal{R}(S,X)C=\mathscr{V}[S,X]-\nabla_SX$. If, furthermore, {\bf (T1)} holds, from Lemma \ref{lemmafundamental}
we obtain $\nabla_SX=\mathscr{V}[S,X]$ and therefore $\mathcal{R}(S,X)C=0$.
\end{proof}
\begin{corollary}\label{corollarycurvendospray}
The curvature endomorphisms ${\bf R}_w$, for $w\in TM\backslash 0$, do not depend on the choice of the lift $\nabla$, but only on the spray $S$. Furthermore, if {\bf (T1)} holds then \textsl{for any} lift $X$ of $u$ at $w$,
\begin{equation}\label{remarkcurvature}
{\bf R}_w(u)=i_w\hspace{0.01cm}^{-1}\mathcal{R}(S,X)C.
\end{equation}
\end{corollary}

\subsection{The associated family of affine connections on $M$}\label{affineconnectionsub}

\begin{definition}\label{definaffine}
Let $w\in TM\backslash 0$. If $\mathcal{O}\subseteq M$ is an open neighborhood of $x=\pi(w)$, we define 
\begin{equation}\nonumber
D^w_\cdot\cdot:T_xM\times\mathfrak{X}(\mathcal{O})\rightarrow T_xM~,~~D^w_uV=i_w\hspace{0.01cm}^{-1}\nabla_{u^\mathfrak{h}}V^\mathfrak{v}, 
\end{equation}
where $u^\mathfrak{h}$ is the horizontal lift of $u$ at $w$.
\end{definition}

\noindent The map $D^w_\cdot\cdot:T_xM\times\mathfrak{X}(\mathcal{O})\rightarrow T_xM$ just defined is $\mathbb{R}$-linear in $u$ and $V$ and satisfies the Leibniz rule
for $V$. In particular, for each smooth curve $\lambda:I\rightarrow M$ through $x$ at $t=t_0$ it induces a map $D^w/dt:\mathfrak{X}(\lambda)\rightarrow T_xM$
satisfying the properties of a covariant derivative. In terms of $\nabla$,
\begin{equation}\label{covariantderivative}
\frac{D^w V}{dt}=i_w\hspace{0.01cm}^{-1}(\frac{\nabla V^\mathfrak{v}}{dt})(t_0),
\end{equation}
where $V^\mathfrak{v}(t)$ is the vertical lift of $V(t)$ along the horizontal lift
$\overline{\lambda}:I\rightarrow TM\backslash 0$ of $\lambda:I\rightarrow M$ through $w$ at $t=t_0$.\\\\
By considering $W\in\mathfrak{X}(\mathcal{O})$ (resp., $W(t)\in\mathfrak{X}(\lambda)$, for some smooth curve $\lambda$) nowhere null, we thus obtain well-defined maps
\begin{equation}\nonumber
\begin{array}{cccc}
D^W_{\cdot}\cdot: & \mathfrak{X}(\mathcal{O})\times\mathfrak{X}(\mathcal{O}) & \rightarrow & \mathfrak{X}(\mathcal{O})\\
D^W/dt: & \mathfrak{X}(\lambda) & \rightarrow & \mathfrak{X}(\lambda)\\
\end{array}
\end{equation}
satisfying the properties of being an affine connection on $\mathcal{O}$ and a covariant derivative along $\lambda$, respectively. 

\begin{definition}
We call $\{D^W\}_{(\mathcal{O},W)}$ the family of affine connections corresponding to $\nabla$.
\end{definition}

\begin{proposition}\label{propfundamental}
Let $W\in\mathfrak{X}(\mathcal{O})$, nowhere null, and $\lambda:I\rightarrow M$ a regular curve on $M$.
\begin{enumerate}
\item If {\bf (T1)} holds, then the map $D^W_W\cdot:\mathfrak{X}(\mathcal{O})\rightarrow\mathfrak{X}(\mathcal{O})$ \textsl{does not depend on the choice of $\nabla$, but only on the spray $S$};
equivalently, if {\bf (T1)} holds, the map $D^{\dot{\lambda}}/dt:\mathfrak{X}(\lambda)\rightarrow\mathfrak{X}(\lambda)$ \textsl{is intrinsic to the spray $S$}.
\item If {\bf (T3)} holds, then we could have chosen any lift of $u$ at $w$ in Definition \ref{definaffine} (resp., any lift of $\lambda$ through $w$ at $t=t_0$ in (\ref{covariantderivative})).
\end{enumerate}
\end{proposition}
\begin{proof}
(1) Along $W(\mathcal{O})\subset TM\backslash 0$, $W^\mathfrak{h}$ coincides with $S$. So, $D^W_WU=i_W\hspace{0.01cm}^{-1}\nabla_SU^\mathfrak{v}$ which according to (2) of Lemma \ref{lemmafundamental},
does not depend on the choice of $\nabla$. For (2), we have to show that $\nabla_X V^\mathfrak{v}=0$ if $X$ is vertical and {\bf (T3)} holds. Choosing $X\in\mathfrak{X}(TM\backslash 0)$ vertical in ${\rm T}(X,Y)=0$
we obtain $\nabla_X\mathscr{J}(Y)=\mathscr{J}[X,Y]$.
If, furthermore, $Y$ is projectable, than $[X,Y]$ is vertical and hence $\mathscr{J}[X,Y]=0$. So, for $X$ vertical and $Y$ projectable, $\nabla_X\mathscr{J}(Y)=0$. Now, just note that
$U^\mathfrak{v}=\mathscr{J}(Y)$ for any lift $Y$ of $U$.
\end{proof}

We leave to $\S$\ref{liftssub} the proof of the following.

\begin{lemma}\label{lemmageo}
A regular curve $\lambda:I\rightarrow M$ is a geodesic of $S$ if, and only if, $\frac{D^{\dot{\lambda}}\dot{\lambda}}{dt}=0$ for all $t$.
\end{lemma}

\subsubsection{Symmetry}\label{symmetrysub}

\begin{proposition}\label{propsymmetry}
Let $W\in\mathfrak{X}(\mathcal{O})$ be nowhere null.
\begin{enumerate}
\item If {\bf (T1)} holds, then $D^W_WV-D^W_VW=[W,V]$ for all $V\in\mathfrak{X}(\mathcal{O})$.
\item If {\bf (T2)} holds, then $D^W_UV-D^W_VU=[U,V]$ for all $U,V\in\mathfrak{X}(\mathcal{O})$.
\end{enumerate}
\end{proposition}
\begin{proof}
As $U^\mathfrak{v}=\mathscr{J}(U^\mathfrak{h})$ and $V^\mathfrak{v}=\mathscr{J}(V^\mathfrak{h})$, then
$D^W_UV-D^W_VU=i_W\hspace{0.01cm}^{-1}\bigl(\nabla_{U^\mathfrak{h}}\mathscr{J}(V^\mathfrak{h})-\nabla_{V^\mathfrak{h}}\mathscr{J}(U^\mathfrak{h})\bigr)
=i_W\hspace{0.01cm}^{-1}{\rm T}(U^\mathfrak{h},V^\mathfrak{h})+i_W\hspace{0.01cm}^{-1}\mathscr{J}[U^\mathfrak{h},V^\mathfrak{h}]$.
But ${\rm d}\pi[U^\mathfrak{h},V^\mathfrak{h}]=[U,V]$, hence $i_W\hspace{0.01cm}^{-1}\mathscr{J}[U^\mathfrak{h},V^\mathfrak{h}]=[U,V]$. To conclude, just note
that, along $W(\mathcal{O})\subset TM\backslash 0$, $W^\mathfrak{h}=S$.
\end{proof}
\noindent As usual, the above symmetry properties of $D^W$ implies the following symmetry between covariant derivatives
\begin{equation}\label{symmetryeq}
\frac{D^TT}{ds}=\frac{D^TU}{dt}
\end{equation}
{\it whenever {\bf (T1)} holds}; here, $T=\partial H/\partial t$ and $U=\partial H/\partial s$ are the coordinate vector fields along a a smooth variation
$H:(-\varepsilon,\varepsilon)\times I\rightarrow M$ of regular curves (i.e. each curve $t\mapsto H(s,t)$ is regular), and $D^T/dt$ and $D^T/ds$
are, respectively, the covariant derivatives along the curves $t\mapsto H(s,t)$ and $s\mapsto H(s,t)$,
corresponding to $T$.

\subsubsection{Working with non-horizontal lifts}\label{liftssub}
Given $W,U\in\mathfrak{X}(\mathcal{O})$ (resp., $W,U\in\mathfrak{X}(\lambda)$), with $W$ nowhere null, ${\rm d}W(U):\mathcal{O}\rightarrow T(TM\backslash 0)$
(resp., $W:I\rightarrow TM\backslash 0$) is a natural lift of $U$ (resp., of $\lambda$) defined along, and tangent to, the submanifold $W(\mathcal{O})\subset TM\backslash 0$.
It will be desirable to use these lifts, instead of the horizontal ones, to compute $D^W_UV$ and $D^W V/dt$.

\begin{lemma}\label{lemma100}
Within the above notation, we have (recall $\S$\ref{tensorsCsubsub})
\begin{eqnarray}
D^W_UV & = & i_W\hspace{0.01cm}^{-1}\nabla_{{\rm d}W(U)}V^\mathfrak{v}-\mathcal{C}_W(D^W_U W,V).\label{eq1}\\
\frac{D^W V}{dt} & = & i_W\hspace{0.01cm}^{-1}\frac{\nabla V^\mathfrak{v}}{dt}-\mathcal{C}_W\Bigl(\frac{D^W W}{dt},V\Bigr)\label{eq2}
\end{eqnarray}
In particular, as $\mathcal{C}_W(\cdot,W)=0$, then $D^W_U W=i_W\hspace{0.01cm}^{-1}\nabla_{{\rm d}W(U)}W^\mathfrak{v}$.
(In (\ref{eq2}), $V^\mathfrak{v}$ stands for the vertical lift of $V\in\mathfrak{X}(\lambda)$ along the curve $W:I\rightarrow TM\backslash 0$.)
\end{lemma}
\begin{proof}
We will only consider (\ref{eq1}), (\ref{eq2}) being analog.
Since $[\mathscr{J}(X),Y]$ is vertical if $Y$ is projectable, we obtain from (\ref{eq3}) that $\nabla_{\mathscr{J}(X)}\mathscr{J}(Y)=\mathscr{C}(X,Y)$ whenever $Y$ is projectable. It follows that
$\nabla_{\mathscr{V}X}V^\mathfrak{v}=\mathscr{C}(\hat{X},V^\mathfrak{h})$ for all $\hat{X}$ such that $\mathscr{J}(\hat{X})=\mathscr{V}X$. Applying this to $X={\rm d}W(U)$, and
noting that $\mathscr{H}{\rm d}W(U)=U^\mathfrak{h}$, we obtain $\nabla_{{\rm d}W(U)}V^\mathfrak{v} = \nabla_{\mathscr{H}{\rm d}W(U)}V^\mathfrak{v} + \nabla_{\mathscr{V}{\rm d}W(U)}V^\mathfrak{v}=
i_WD^W_UV + \mathscr{C}(\hat{X},V^\mathfrak{h})$. Now, from $\mathscr{J}(\hat{X})=\mathscr{V}{\rm d}W(U)$ we obtain ${\rm d}\pi(\hat{X})=i_W\hspace{0.01cm}^{-1}\mathscr{V}{\rm d}W(U)$ and, therefore,
$\mathscr{C}(\hat{X},V^\mathfrak{h})=i_W\mathcal{C}_W(i_W\hspace{0.01cm}^{-1}\mathscr{V}{\rm d}W(U),V)$. It remains to show that
\begin{equation}\label{eq500}
i_W\hspace{0.01cm}^{-1}\mathscr{V}{\rm d}W(U)=D^W_UW.
\end{equation}
For such, let $V=W$ in $\nabla_{{\rm d}W(U)}V^\mathfrak{v} = i_WD^W_UV + \mathscr{C}(\hat{X},V^\mathfrak{h})$ and observe that $\nabla_{{\rm d}W(U)}W^\mathfrak{v}=\mathscr{V}{\rm d}W(U)$, since
$W^\mathfrak{v}|_{W(\mathcal{O})}=C|_{W(\mathcal{O})}$, and $\mathscr{C}(\hat{X},W^\mathfrak{h})=0$ since $W^\mathfrak{h}|_{W(\mathcal{O})}=S|_{W(\mathcal{O})}$.
\end{proof}
\begin{remark}\label{remark300}
For future reference, we remark that, analogously to (\ref{eq500}), we have $D^WW/dt={i_W}^{-1}\mathscr{V}dW/dt$, for any nowhere null $W\in\mathfrak{X}(\lambda)$.
\end{remark}

\begin{proof}[Proof of Lemma \ref{lemmageo}]
The vertical lift $\dot{\lambda}^\mathfrak{v}$ of $\dot{\lambda}$ along the curve $t\mapsto\dot{\lambda}(t)$ coincides with $C$, so
substituting $W(t)=V(t)=\dot{\lambda}(t)$ in (\ref{eq2}), and recalling that $\mathcal{C}_{\dot{\lambda}}(\cdot,\dot{\lambda})=0$, we obtain $D^{\dot{\lambda}}\dot{\lambda}/dt=
i_{\dot{\lambda}}\hspace{0.01cm}^{-1}\nabla\dot{\lambda}^{\frak{v}}/dt=i_{\dot{\lambda}}\hspace{0.01cm}^{-1}\nabla C/dt$ which is null if, and only if, the curve $t\mapsto\dot{\lambda}(t)$
is horizontal, which in turn is equivalent to it being an integral line of $S$.

\end{proof}

\subsubsection{The curvature endomorphisms of the affine connections}\label{curvatureendomsub}
We can now prove the main result of $\S$2, which shows how one can compute the curvature endomorphisms, defined in $\S$\ref{subsubcurvspray}, by means of
the covariant derivatives.

\begin{theorem}\label{theoremcurvature}
Suppose that {\bf (T1)} holds. Let $H:(-\varepsilon,\varepsilon)\times I\rightarrow M$ be a smooth variation of a curve $\lambda(t)=H(0,t)$ through regular curves (i.e.,
each $t\mapsto H(s,t)$ is regular). Following the same notation as $\S$\ref{symmetrysub}, if $\lambda$ is a geodesic then, at any $(0,t)$,
\begin{equation}\nonumber
\frac{D^T}{ds}\frac{D^TT}{dt}-\frac{D^T}{dt}\frac{D^TT}{ds}={\bf R}_T(U).
\end{equation}
In particular, if the variation is through geodesics, the variational vector field $J(t)=U(0,t)$ will satisfy the \textsl{Jacobi equation}
\begin{equation}\nonumber
\frac{{D^{\dot{\lambda}}}^2J}{dt^2}+{\bf R}_{\dot{\lambda}}(J)=0
\end{equation}
\end{theorem}

Instead of proving this theorem, we will prove the following somewhat more abstract version from which Theorem \ref{theoremcurvature} follows easily.

\begin{proposition}\label{propositioncurvature}
Suppose that {\bf (T1)} holds. Let $W\in\mathfrak{X}(\mathcal{O})$, nowhere null, be such that its integral curve through some point $x_0$ is a geodesic.
Then, at $x_0$, the curvature endomorphism in the direction $W$ of the affine connection $D^W$ coincides with ${\bf R}_W$; that is, given
$U\in\mathfrak{X}(\mathcal{O})$, then, at $x_0$,
\begin{equation}\nonumber
D^W_UD^W_WW-D^W_WD^W_UW+D^W_{[W,U]}W={\bf R}_W(U)
\end{equation}
\end{proposition}
\begin{proof}
To simplify the notation, let $\widetilde{W}={\rm d}W(W)$ and $\widetilde{U}={\rm d}W(U)$ be the vector fields tangent to
$W(\mathcal{O})\subset TM\backslash 0$.
According to Lemma \ref{lemma100}, along $W(\mathcal{O})$ we have $(D^W_W W\bigr)^\mathfrak{v}=\nabla_{\widetilde{W}}W^\mathfrak{v}$. Hence, applying
Lemma \ref{lemma100} again,
\begin{equation}\label{eq5}
 D^W_UD^W_W W=i_W\hspace{0.01cm}^{-1}\nabla_{\widetilde{U}}\nabla_{\widetilde{W}}W^\mathfrak{v}-\mathcal{C}_W(D^W_U W,D^W_W W).
\end{equation}
Arguing similarly,
\begin{equation}\label{eq6}
D^W_W D^W_U W = i_W\hspace{0.01cm}^{-1}\nabla_{\widetilde{W}}\nabla_{\widetilde{U}}W^\mathfrak{v}-\mathcal{C}_W(D^W_W W,D^W_UW).
\end{equation}
Also, from ${\rm d}W([W,U])=[\widetilde{W},\widetilde{U}]$ and Lemma \ref{lemma100}, we have $D^W_{[W,U]}W=i_W^{-1}\nabla_{[\widetilde{W},\widetilde{U}]}W^\mathfrak{v}$.
Together with (\ref{eq5}) and (\ref{eq6}), this gives us
\begin{eqnarray}
D^W_UD^W_WW-D^W_WD^W_UW+D^W_{[W,U]}W & = & i_W\hspace{0.01cm}^{-1}\mathcal{R}\bigl(\widetilde{W},\widetilde{U}\bigr)W^\mathfrak{v}\nonumber\\
& + & \mathcal{C}_W(D^W_W W,D^W_UW)-\mathcal{C}_W(D^W_U W,D^W_W W).\nonumber
\end{eqnarray}
Now, $W^\mathfrak{v}|_{W(\mathcal{O})}=C$, and the hypothesis on $W$ implies that, at $W(x_0)$, $\widetilde{W}=S$. Therefore, as {\bf (T1)} holds,
and ${\rm d}\pi(\widetilde{U})=U$, relation (\ref{remarkcurvature}) guaranties that, at $x_0$,
$i_W\hspace{0.01cm}^{-1}\mathcal{R}\bigl(\widetilde{W},\widetilde{U}\bigr)W^\mathfrak{v}={\bf R}_W(U)$. To conclude, just note that the hypothesis on $W$
implies $(D^W_WW)(x_0)=0$.
\end{proof}

\section{The case of a Finsler metric}\label{sectionfinsler}
We now turn back to the case where we are given a Finsler metric $F$ on $M$. As it is well-known (see \cite{marsden}), the solutions to the Euler-Lagrange equations of (\ref{energyfunctional2}) 
are the geodesics of a second order differential equation $S$ on $M$ defined via
\begin{equation}\nonumber
{\rm d}{\rm E}=\omega_F(S\hspace{0.03cm},\hspace{0.03cm}\cdot), 
\end{equation}
where $\omega_F$ is the symplectic structure on $TM\backslash 0$ induced by $F$ (see the Appendix). The homogeneity
of $F$ guarantees that $S$ is in fact a spray, the so called {\it geodesic spray} of $F$.
\par\textsl{Throughout this section, let be fixed
a lift $\nabla$ of $\Gamma_S$ satisfying condition {\bf (T1)}}.

\subsection{Metric conditions on $\nabla$}
By abuse of notation we still denote by $g$ the Riemannian metric on (\ref{verticalbundle}) induced by the fundamental tensor of $F$ with the help of the vertical lift.
We start by observing that $g$ is always parallel in the direction of $S$:
\begin{lemma}\label{lemma200}
We have $\nabla_Sg=0$.
\end{lemma}
\begin{proof}
We already know from Lemma \ref{lemmafundamental} that this covariant derivative does not depend on the choice of $\nabla$. In particular, we can assume that $\nabla$ is
any one of the classical connections discussed in $\S$\ref{sectionclassicalconnections}, in which case the result follows immediately from
the properties of the connection.
\end{proof}

\begin{definition}
We denote by $\mathscr{C}_\flat$ the following metric contraction of $\mathscr{C}$,
\begin{equation}\nonumber
\mathscr{C}_\flat(X,Y,Z)=g(\mathscr{C}(X,Y),\mathscr{J}(Z)).
\end{equation}
\end{definition}
\noindent Also, for each $w\in TM\backslash 0$, we define
$(\mathcal{C}_\flat)_w:T_{\pi(w)}M\times T_{\pi(w)}M\times T_{\pi(w)}M\rightarrow\mathbb{R}$ by
$(\mathcal{C}_\flat)_w(u,v,t)=i_w\hspace{0.01cm}^{-1}\mathscr{C}_\flat(X,Y,Z)=g_w(\mathcal{C}_w(u,v),t)$,
where $X,Y,Z\in T_w TM$ are any lifts of $u,v,t$, respectively.\\\\
The metric conditions on $\nabla$ we will be interested in can now be stated as
\begin{eqnarray}
{\bf (M1)} & & (\nabla_\cdot g)(\cdot,C)=0.\nonumber\\
{\bf (M2)} & & \mathscr{C}_\flat(\cdot,\cdot,S)=0.\nonumber
\end{eqnarray}
Note that condition {\bf (M2)} may be written as $(\mathcal{C}_\flat)_w(\cdot,\cdot,w)=0$ for all $w\in TM\backslash 0$.

\subsection{Metric properties of the affine connections}
Back to the family of affine connections $D^W$, we can state
\begin{proposition}\label{prop300}
Let $W,U,V,T\in\mathfrak{X}(\mathcal{O})$, with $W$ nowhere null, where $\mathcal{O}\subseteq M$ is an open set.
\begin{itemize}
\item[$(i)$] If {\bf (M1)} and {\bf (M2)} hold, then
\begin{equation}\nonumber
Ug_W(W,V)=g_W(D^W_UW,V)+g_W(W,D^W_UV).
\end{equation}
\item[$(ii)$] If the integral curve of $W$ through some point $x_0$ is a geodesic, then, at
$x_0$,
\begin{equation}\nonumber
Wg_W(T,V)=g_W(D^W_WT,V)+g_W(T,D^W_WV).
\end{equation}
\end{itemize}
\end{proposition}
\begin{proof}
Computing $\nabla_{{\rm d}W(U)}T^\mathfrak{v}$ and $\nabla_{{\rm d}W(U)}V^\mathfrak{v}$ with the help of Lemma \ref{lemma100}, 
\begin{eqnarray}
(\nabla_{{\rm d}W(U)}g)(T^\mathfrak{v},V^\mathfrak{v})|_{W(x)} & = & (D^W_Ug_W)(T,V)|_x
-(\mathcal{C}_\flat)_W(D^W_UW,T,V)|_x\label{eq630}\\
& - & (\mathcal{C}_\flat)_W(D^W_UW,V,T)|_x\nonumber 
\end{eqnarray}
$(i)$ Let $T=W$ in (\ref{eq630}). As $W^\mathfrak{v}|_{W(\mathcal{O})}=C|_{W(\mathcal{O})}$, the left-hand side of (\ref{eq630}) vanishes since {\bf (M1)} holds. Also, (2) of Lemma \ref{lemmaC} gives
$(\mathcal{C}_\flat)_W(D^W_UW,T,V)=0$,
whereas the hypothesis {\bf (M2)} gives $(\mathcal{C}_\flat)_W(D^W_UW,V,T)=0$.
\\
$(ii)$ Let $U=W$ in (\ref{eq630}). The hypothesis on $W$ implies ${\rm d}W(W)|_{W(x_0)}=S|_{W(x_0)}$ and hence, from Lemma \ref{lemma200}, 
$(\nabla_{{\rm d}W(U)}g)(T^\mathfrak{v},V^\mathfrak{v})|_{W(x_0)}=0$. Also, $(D^W_WW)(x_0)=0$ and the result follows.

\end{proof}

\begin{remark}
\begin{enumerate}
\item Observe that although we have assumed the validity of {\bf (T1)} throughout this section, property $(i)$ of Proposition \ref{prop300}
does not require that hypothesis.
\item As usual, Proposition \ref{prop300} implies analogous statements for covariant derivatives which, for the sake of brevity, we will not
enunciate here.
\end{enumerate}
\end{remark}

\subsection{The second variation of energy and the second fundamental form}\label{subsectionsecondvariation}

\textsl{Throughout this subsection, we assume that conditions {\bf (M1)} and {\bf (M2)} hold.}\\\\
Let be fixed two submanifolds $\mathcal{P}_1,\mathcal{P}_2\subset M$ and denote by $\Omega_{\mathcal{P}_1,\mathcal{P}_2}$ the set of all
regular curves $\lambda:I=[0,1]\rightarrow M$ joining $\mathcal{P}_1$ to $\mathcal{P}_2$. The corresponding energy functional is
\begin{equation}\label{energyfunctional}
{\rm E}~:~\Omega_{\mathcal{P}_1,\mathcal{P}_2}\longrightarrow\mathbb{R}~,~~~{\rm E}(\lambda)=\frac{1}{2}\int_0^1F(\dot{\lambda}(t))^2dt.
\end{equation}

When computing the second variation of (\ref{energyfunctional}) at a critical point, a type of second fundamental form will come into play. Firstly, for a submanifold $\mathcal{P}\subset M$
we define
\begin{definition}\
Given $x\in\mathcal{P}$, we say that a vector $\eta\in T_xM\backslash 0$ is {\it normal} to $\mathcal{P}$ if $g_\eta(\eta,u)=0$ for all $u\in T_x\mathcal{P}$.
The set of all such vectors forms a cone in $T_xM\backslash 0$, called {\it the normal cone to $\mathcal{P}$ at $x$}, and denoted by $\nu_x(\mathcal{P})$.
\end{definition}

\begin{definition}
Given $\eta\in\nu_x(\mathcal{P})$, the {\it second fundamental form of $\mathcal{P}$ relative to the normal direction $\eta$} is the map
\begin{equation}\nonumber
{\rm h}^\mathcal{P}_\eta:T_x\mathcal{P}\times T_x\mathcal{P}\rightarrow\mathbb{R}~,~~{\rm h}^\mathcal{P}_\eta(u,v)=\frac{1}{2}g_\eta\bigl(\eta\hspace{0.05cm},\hspace{0.05cm}D^\eta_UV+D^\eta_VU\bigr),
\end{equation}
where $U,V\in\mathfrak{X}(\mathcal{P})$ are any extensions of $u,v$. Of course, this is a well-defined symmetric bilinear form.
\end{definition}

\begin{remark}
\begin{enumerate}
\item We remark that, as a consequence of Proposition \ref{propfundamental} and the theorem below, ${\rm h}^\mathcal{P}_\eta$ does not depend on the choice of $\nabla$ as long as
conditions {\bf (T1)}, {\bf (M1)} and {\bf (M2)} hold. Indeed, this is also the case even if we drop assumption {\bf (T1)} as we show in Proposition \ref{propfundamentalform}.
\item If, in addition, {\bf (T2)} holds, then it follows from Proposition \ref{propsymmetry} that ${\rm h}_\eta^\mathcal{P}(u,v)=g_\eta(\eta,D^\eta _UV)$.
\end{enumerate}
\end{remark}

\begin{theorem}
Let $s\mapsto\lambda_s\in\Omega_{\mathcal{P}_1,\mathcal{P}_2}$ be a smooth 1-parameter family starting at a geodesic $\lambda=\lambda_0$ normal to $\mathcal{P}_1$ and $\mathcal{P}_2$
at its extremes. Denoting by $V$ the variational vector field along $\lambda$, we have
\begin{eqnarray}
\frac{d^2}{ds^2}\Big|_{s=0}{\rm E}(\lambda_s) & = &  \int_0^1g_{\dot{\lambda}}\Bigl(\frac{D^{\dot{\lambda}}V}{dt},\frac{D^{\dot{\lambda}}V}{dt}\Bigr)-
g_{\dot{\lambda}}\bigl({\bf R}_{\dot{\lambda}}(V),V\bigr)\hspace{0.05cm}dt\label{eq700}\\
& & +~{\rm h}^{\mathcal{P}_2}_{\dot{\lambda}(1)}(V(1),V(1))-{\rm h}^{\mathcal{P}_1}_{\dot{\lambda}(0)}(V(0),V(0)).\nonumber
\end{eqnarray}
\end{theorem}
\begin{proof}
Once we have at hand Proposition \ref{prop300}, Theorem \ref{theoremcurvature} and the symmetry relation (\ref{symmetryeq}), the proof proceeds exactly as in the Riemannian case
since the homogeneity of $F$ implies the relation $F(w)^2=g_w(w,w)$. We only remark that at some point it will be necessary to assure that
\begin{equation}\nonumber
\frac{\partial}{\partial s}g_T\Bigl(\frac{D^TT}{dt},U\Bigr)\Big|_{s=0}=g_T\Bigl(\frac{D^T}{ds}\frac{D^TT}{dt},U\Bigr)\Big|_{s=0}.
\end{equation}
(following the notation in Theorem \ref{theoremcurvature}). Now, this equality is true regardless of any assumptions on $\nabla$ as follows from the fact that
$(D^TT/dt)(0,t)=0$ for all $t$ and the following trivial fact that just reflects the tensorial nature of $D^Wg_W$:
\\\\
{\it Given $W,U,V,T\in\mathfrak{X}(\mathcal{O})$, if $V(x_0)=0$ then, at $x_0$, $Tg_W(V,U)=g_W(D^W_TV,U)$}
\end{proof}

\section{The connections of Berwald, Cartan, Chern-Rund and Hashiguchi}\label{sectionclassicalconnections}

\subsection{The Berwald connection of a spray}\label{subberwald}
The Berwald connection is defined for any spray $S$ regardless to whether or not $S$ is the geodesic spray of a Finsler metric.
We take from \cite{grifone2} the following expressions.

\begin{propdef}
The following
\begin{eqnarray}
\nabla^*_{\mathscr{J}(X)}\mathscr{J}(Y) & = & \mathscr{J}[\mathscr{J}(X),Y]\nonumber\\
\nabla^*_{\mathscr{H}X}\mathscr{J}(Y) & = & \mathscr{V}[\mathscr{H}X,\mathscr{J}(Y)]\nonumber
\end{eqnarray}
defines a lift $\nabla^*$ of $\Gamma_S$, called the \textsl{Berwald connection of $S$}.
\end{propdef}

\noindent Let $\mathscr{F}$ denote the almost-complex structure
on $TM\backslash 0$ characterized by $\mathscr{F}\mathscr{J}=\mathscr{H}$ and
$\mathscr{F}\mathscr{H}=-\mathscr{J}$. Having in mind (\ref{eq3}), a direct computation shows that any linear connection $\nabla$ on (\ref{verticalbundle})
may be written as
\begin{equation}\nonumber
\nabla_X\mathscr{J}(Y)=\nabla^*_X\mathscr{J}(Y)+\mathscr{C}(\mathscr{F}X,Y)+\mathscr{C}'(X,Y), 
\end{equation}
where $\mathscr{C}'(X,Y)=\nabla_{\mathscr{H}X}\mathscr{J}(Y)-\mathscr{V}[\mathscr{H}X,\mathscr{J}(Y)]$ defines a semi-basic $\mathscr{V}TM$-valued tensor field
on $TM\backslash 0$. Actually, one can show the following

\begin{proposition}
There is a 1-1 correspondence between the lifts of $\Gamma_S$ and the pairs of $\mathscr{V}TM$-valued semi-basic tensor fields $\mathscr{C}$ and $\mathscr{C}'$
on $TM\backslash 0$ satisfying $\mathscr{C}(\cdot,S)=\mathscr{C}'(\cdot,S)=0$.
\end{proposition}

\subsection{The connections of Berwald (bis), Cartan, Chern-Rund and Hashiguchi}
Now, let us suppose that $S$ is the geodesic spray of a Finsler metric $F$.
\begin{definition}
For each $w\in TM\backslash 0$, the {\it Cartan tensor} of $F$ at $w$ is the 3-linear map ${\rm C}_w:T_{\pi(w)}M\times T_{\pi(w)}M\times T_{\pi(w)}M\rightarrow\mathbb{R}$,
\begin{equation}\nonumber
{\rm C}_w(u,v,z)=\frac{1}{4}\frac{\partial^3}{\partial r\partial s\partial t}\Big|_{r=s=t=0}F(w+ru+sv+tz)^2.
\end{equation}
\end{definition}
\noindent Next we define the tensor ${\rm C}'$. For that, let $D^W$ be the affine connections associated to any lift $\nabla$ of $\Gamma_S$ satisfying {\bf (T1)} (e.g. the Berwald lift).
\begin{definition}
Given $w\in TM\backslash 0$, we define a 3-linear map ${\rm C}'_w:T_{\pi(w)}M\times T_{\pi(w)}M\times T_{\pi(w)}M\rightarrow\mathbb{R}$ as 
\begin{equation}\nonumber
{\rm C}_w'(u,v,z)=(D^W_W{\rm C}_W)|_{\pi(w)}(u,v,z),
\end{equation}
where $W$ is any nowhere null vector field defined around $\pi(w)$ whose integral curve through $\pi(w)$ is a geodesic.
\end{definition}

\begin{lemma}\label{cartanlemma}
For each $w\in TM\backslash 0$, ${\rm C}_w$ and ${\rm C}'_w$ are fully symmetric and ${\rm C}_w(\cdot,\cdot,w)={\rm C}'_w(\cdot,\cdot,w)=0$. 
\end{lemma}
\begin{proof}
The statements about ${\rm C}$ are well-known (e.g. \cite{chern}), whereas the ones about ${\rm C}'$ follow easily from those. 
\end{proof}

Let us consider the following conditions for a linear connection $\nabla$ on (\ref{verticalbundle})
\begin{eqnarray}
{\bf (M3)} & & (\nabla_X g)(\cdot,\cdot)=2{\rm C}(\mathscr{V}X,\hspace{0.05cm}\cdot\hspace{0.05cm},\hspace{0.05cm}\cdot\hspace{0.05cm})\nonumber\\
{\bf (M4)} & & (\nabla_X g)(\cdot,\cdot)=2{\rm C}'(\mathscr{J}(X),\hspace{0.05cm}\cdot\hspace{0.05cm},\hspace{0.05cm}\cdot\hspace{0.05cm})\nonumber\\
{\bf (M5)} & & (\nabla_X g)(\cdot,\cdot)=2{\rm C}(\mathscr{V}X,\hspace{0.05cm}\cdot\hspace{0.05cm},\hspace{0.05cm}\cdot\hspace{0.05cm})+
2{\rm C}'(\mathscr{J}(X),\hspace{0.05cm}\cdot\hspace{0.05cm},\hspace{0.05cm}\cdot\hspace{0.05cm})\nonumber\\
{\bf (M6)} & & \nabla g = 0\nonumber\\
{\bf (M7)} & & \mathscr{C}_\flat(X,Y,Z)=\mathscr{C}_\flat(X,Z,Y),~~\mbox{for all $X,Y,Z$},\nonumber
\end{eqnarray}
where we have tacitly identified ${\rm C}$ and ${\rm C}'$ with tensor fields on (\ref{verticalbundle}) via the vertical lift. It follows from Lemma \ref{cartanlemma} that
each one of {\bf (M3)}, {\bf (M4)}, {\bf (M5)} implies {\bf (M1)}, whereas (2) of Lemma \ref{lemmaC} guaranties that {\bf (M7)} implies {\bf (M2)}.
\\\\
The connections of Berwald, Cartan, Chern-Rund and Hashiguchi can be described by the following existence-uniqueness results. 
We remark that, except for the statement concerning the Cartan connection (which is stated without proof in \cite{zadeh68}), we were not able to
find in the literature an appropriate reference for them, although one can show without difficult their equivalence with 
the statements in \cite{bucataru}.
\begin{theorem}\label{theoremclassicallifts}
For each set of conditions below, there exists one, and only one, lift $\nabla$ of $\Gamma_S$ which fulfils them:
\begin{enumerate}
\item {\bf (T3)} and {\bf (M5)}.
\item {\bf (T2)}, {\bf (M6)} and {\bf (M7)}.
\item {\bf (T3)} and {\bf (M3)}.
\item {\bf (T2)}, {\bf (M4)} and {\bf (M7)}.
\end{enumerate}
They are called, respectively, the Berwald, Cartan, Chern-Rund and Hashiguchi connections of the Finsler metric $F$.
\end{theorem}

\begin{remark}\label{remarkclassicallifts}
Recall from $\S$\ref{subberwald} that any $\nabla$ is determined by the tensors $\mathscr{C}$ and $\mathscr{C}'$.  
In terms of these, or rather in terms of  
$\mathcal{C}_\flat$ and $\mathcal{C}'_\flat$ ($\mathcal{C}'_\flat$ is defined analogously to $\mathcal{C}_\flat$), the above connections are given by (compare with \cite{crampin})
\begin{enumerate}
\item \textsl{Cartan}: $(\mathcal{C}_\flat)_w(u,v,t)={\rm C}_w(u,v,t)$ and $(\mathcal{C}'_\flat)_w(u,v,t)={\rm C}'_w(u,v,t)$;
\item \textsl{Chern-Rund}: $\mathcal{C}_\flat=0$ and $(\mathcal{C}'_\flat)_w(u,v,t)={\rm C}'_w(u,v,t)$;
\item \textsl{Hashiguchi}: $(\mathcal{C}_\flat)_w(u,v,t)={\rm C}_w(u,v,t)$ and $\mathcal{C}'_\flat=0$.
\end{enumerate}
\end{remark}

As for the families of affine connections on $M$ corresponding to the above classical connections, we have the following theorem which
was already known in the cases of Berwald and Chern-Rund connections (e.g. \cite{shen}).
\begin{theorem}\label{thmclassicalaffineconnections}
The Cartan and Chern-Rund families of affine connections satisfy both 
\begin{eqnarray}
D^W_UV-D^W_VU & = & [U,V]\label{eq900}\\
(D^W_Ug_W)(T,V) & = & 2{\rm C}_W(D^W_UW,T,V), 
\end{eqnarray}
whereas the ones corresponding to Berwald and Hashiguchi satisfy both $(\ref{eq900})$ and
\begin{eqnarray}
(D^W_Ug_W)(T,V) & = & 2{\rm C}_W(D^W_WU,T,V)+2{\rm C}'_W(U,T,V).
\end{eqnarray}
\end{theorem}
\begin{proof}
The symmetry (\ref{eq900}) was already proved in Proposition \ref{propsymmetry} since all the connections satisfy {\bf (T2)}.
Now, on one hand, formula (\ref{eq630}) express $D^Wg_W$ in terms of $\nabla g$ and $\mathcal{C}_\flat$. On the the other hand,
Theorem \ref{theoremclassicallifts} provides $\nabla g$, whereas Remark \ref{remarkclassicallifts} gives us the corresponding tensors $\mathcal{C}_\flat$.
These observations, together with (\ref{eq500}), immediatly lead to the desired relations. 
\end{proof}

\begin{proposition}
Each set of properties stated in Theorem \ref{thmclassicalaffineconnections} \textsl{uniquely} determine a family of affine connections on $M$. 
Therefore, the Cartan and Chern-Rund (resp., Berwald and Hashiguchi) families of affine connections coincide. 
\end{proposition}
\begin{proof}
Both uniqueness statements are proved in \cite{shen}. For an alternative proof regarding the first set of conditions, see \cite{miguel}. 
\end{proof}

\section{Appendix: The second fundamental form from a symplectic point of view}
The notion of second fundamental form we considered in $\S$\ref{subsectionsecondvariation} can be regarded from a purely symplectic point of view.
First of all, the Legendre transformation of $F$,
$\mathscr{L}:TM\backslash 0\rightarrow T^*M\backslash 0$, $\mathscr{L}(w)\cdot u=g_w(w,u)$,
pulls the canonical symplectic form of $T^*M$ back to a symplectic form $\omega_F$ on $TM\backslash 0$ which can be described via the isomorphism
\begin{equation}\label{decomposition2}
T_wTM\simeq T_{\pi(w)}M\oplus T_{\pi(w)}M~,~~X\mapsto\bigl({\rm d}\pi(X),{i_w}^{-1}\mathscr{V}X\bigr)
\end{equation}
as $\omega_F\bigl((u_1,v_1),(u_2,v_2)\bigr)=g_w(u_1,v_2)-g_w(v_1,u_2)$.
\par Given a submanifold $\mathcal{P}\subset M$, let $\nu(\mathcal{P})\subset TM\backslash 0$ denote its {\it normal bundle}, i.e. the collection of
all normal cones $\nu_x(\mathcal{P})$, $x\in\mathcal{P}$.

\begin{lemma}
$\nu(\mathcal{P})$ is a Lagrangean submanifold of $(TM\backslash 0,\omega_F)$.
\end{lemma}
\begin{proof}
Observe that, via the Legendre transformation $\mathscr{L}$, $\nu(\mathcal{P})$ corresponds to the conormal bundle ${\rm co}(\mathcal{P})=\{\xi\in T^*M:
\pi(\xi)\in\mathcal{P}~{\rm and}~\xi(T_{\pi(\xi)}\mathcal{P})=\{0\}\}$ with the zero section excluded. Now, ${\rm co}(\mathcal{P})$ is a Lagrangean submanifold of
$T^*M$ since ${\rm dim~  co}(\mathcal{P})={\rm dim}M$ and the canonical 1-form of $T^*M$ pulls back to the null form
on ${\rm co}(\mathcal{P})$ as can be easily verified.
\end{proof}
We refer to \cite[Exercise 1.17]{tausk} for the following result from symplectic linear algebra.
\begin{proposition}
Let $(V,\langle,\rangle)$ be a finite dimensional real vector space endowed with a non-degenerate inner product. On $V\oplus V$, consider the symplectic
bilinear form $\omega\bigl((u_1,v_1),(u_2,v_2)\bigr)=\langle u_1,v_2\rangle-\langle v_1,u_2\rangle$. Then, there is a bijection between the set of all pairs
$(S,{\rm b})$, with $S\subseteq V$ a linear subspace and {\rm b} a symmetric bilinear form on $S$, and the set of all Lagrangean subspaces $L\subset(V\oplus V,\omega)$,
which maps $(S,{\rm b})$ to $L=\{(u,v)\hspace{0.05cm}:\hspace{0.05cm}u\in S~{\rm and}~\langle v,\cdot\rangle|_S+{\rm b}(u,\cdot)=0\}$.
\end{proposition}
Observe that, in the above proposition, we must have $S={\rm pr}_1(L)$, where ${\rm pr}_1$ is the projection onto the first summand $V$. Hence,
the above proposition applied to
the Lagrangean subspace $T_\eta\nu(\mathcal{P})\subset\bigl(T_xM\oplus T_xM,\omega_F\bigr)$, $x=\pi(\eta)$, gives us a symmetric bilinear form ${\rm b}_\eta$ on
${\rm d}\pi(T_\eta\nu(\mathcal{P}))=T_x\mathcal{P}$ such that
\begin{equation}\label{eq250}
T_\eta\nu(\mathcal{P})=\{(u,v)~:~u\in T_x\mathcal{P}~{\rm and}~g_\eta(v,\cdot)|_{T_x\mathcal{P}}+{\rm b}_\eta(u,\cdot)=0\}.
\end{equation}
\begin{proposition}\label{propfundamentalform}
The form ${\rm b}_\eta$ coincides with the second fundamental form ${\rm h}^\mathcal{P}_\eta$ constructed with the help of any lift $\nabla$ of $\Gamma_S$ satisfying {\bf (M1)} and {\bf (M2)}.
\end{proposition}
\begin{proof}
Given $u\in T_x\mathcal{P}$, we will show that ${\rm b}_\eta(u,u)={\rm h}^\mathcal{P}_\eta(u,u)$. Let $v\in T_xM$ be such that $X:=(u,v)\in T_\eta\nu(\mathcal{P})$, so that,
by (\ref{eq250}), ${\rm b}_\eta(u,u)=g_\eta(v,u)$.
Let $\theta:(-\varepsilon,\varepsilon)\rightarrow\nu(\mathcal{P})$ be any smooth curve with $\theta(0)=\eta$ and $\dot{\theta}(0)=X$. Define
$H:(-\varepsilon,\varepsilon)\times(-\delta,\delta)\rightarrow M$ by $H(s,t)=\gamma_{\theta(s)}(t)$, where $\gamma_{\theta(s)}:(-\delta,\delta)\rightarrow M$ is the geodesic
with $\dot{\gamma}_{\theta(s)}(0)=\theta(s)$. As $s\mapsto H(s,0)$ is a curve on $\mathcal{P}$ with velocity $u$ at $t=0$, and $T(0,0)=\eta$, it follows from the definition of
${\rm h}^\mathcal{P}_\eta$ that ${\rm h}^\mathcal{P}_\eta(u,u)=g_T\bigl(D^TU/ds,T\bigr)|_{(0,0)}$. Applying $(i)$ of Proposition \ref{prop300} and recalling that, by construction,
$g_T(U,T)|_{(s,0)}=
g_{\theta(s)}(U(s,0),\theta(s))=0$ for all $s$, we obtain ${\rm h}^\mathcal{P}_\eta(u,u)=-g_T\bigl(U,D^TT/ds\bigr)|_{(0,0)}=g_\eta\bigl(u,D^\theta\theta/ds|_{s=0}\bigr)$. Now,
from Remark \ref{remark300}, $D^{\theta}\theta/ds|_{s=0}={i_{\theta(0)}}^{-1}\mathscr{V}\dot{\theta}(0)=v$. Therefore, ${\rm h}^\mathcal{P}_\eta(u,u)=g_\eta(u,v)$.
\end{proof}

\section*{Acknowledgments} 
The author would like to thank the support provided by the Mathematischen Instituts der Universit\"at Leipzig, where this work was done, and the financial support
provided by the Brazilian program Science Without Borders, grant No. 232664/2014-5.

\bibliographystyle{amsplain}

\end{document}